\documentclass{amsart}
\usepackage{amsfonts,amssymb,amsmath,amsthm}
\usepackage{url}
\usepackage{enumerate}
\newtheorem{thm}{Theorem}[section]
\newtheorem{cor}[thm]{Corollary}
\newtheorem{lem}[thm]{Lemma}
\newtheorem{prop}[thm]{Proposition}
\theoremstyle{definition}

\theoremstyle{remark}

\newtheorem{exe}[thm]{\bf Example}
\numberwithin{equation}{section}

\begin{document}
\title[Hypercyclic abelian semigroups of affine maps on $\mathbb{C}^{n}$ ]{Hypercyclic abelian semigroups \\ of affine maps on $\mathbb{C}^{n}$ }

\author{Yahya N'dao}

 \address{Yahya N'dao, University of Moncton, Department of mathematics and statistics, Canada}
 \email{yahiandao@yahoo.fr ; yahiandao@voila.fr}



\begin{abstract}

We give a characterization of hypercyclic abelian semigroup
$\mathcal{G}$ of affine maps on $\mathbb{C}^{n}$. If $\mathcal{G}$
is finitely generated, this characterization is explicit. We prove
in particular
 that no abelian group generated by $n$ affine maps on $\mathbb{C}^{n}$ has a dense orbit.
\end{abstract}

\maketitle

\section{ Introduction}

Let $M_{n}(\mathbb{C})$ be the set of all square matrices of order
$n\geq 1$  with entries in  $\mathbb{C}$ and $GL(n, \ \mathbb{C})$
be the group of all invertible matrices of $M_{n}(\mathbb{C})$. A
map $f: \ \mathbb{C}^{n}\longrightarrow \mathbb{C}^{n}$  is called
an affine map if there exist $A\in M_{n}(\mathbb{C})$  and  $a\in
\mathbb{C}^{n}$  such that  $f(x)= Ax+a$,  $x\in \mathbb{C}^{n}$.
We denote  $f= (A,a)$, we call  $A$ the \textit{linear part} of
$f$. The map $f$ is invertible if $A\in GL(n, \mathbb{C})$. Denote
by $MA(n, \ \mathbb{C})$  the vector space of all affine maps on
$\mathbb{C}^{n}$ and $GA(n, \ \mathbb{C})$ the group of all
invertible affine maps of $MA(n, \mathbb{C})$.

Let  $\mathcal{G}$  be an abelian affine sub-semigroup of $MA(n, \
\mathbb{C})$. For a vector  $v\in \mathbb{C}^{n}$,  we consider
the orbit of  $\mathcal{G}$  through  $v$: $ \mathcal{G}(v) =
\{f(v): \ f\in \mathcal{G}\} \subset \mathbb{\mathbb{C}}^{n}$.
Denote by $\overline{E}$
 the closure of a subset $E\subset \mathbb{C}^{n}$.
The group $\mathcal{G}$ is called \textit{hypercyclic} if there
exists a vector $v\in {\mathbb{C}}^{n}$ such that
$\overline{\mathcal{G}(v)}={\mathbb{C}}^{n}$. For an account of
results and bibliography on hypercyclicity, we refer to the book
\cite{bm} by Bayart and Matheron.
\
\\
 Let $n\in\mathbb{N}_{0}$  be fixed, denote by:
\\
\textbullet \ $\mathbb{C}^{*}= \mathbb{C}\backslash\{0\}$,
$\mathbb{R}^{*}= \mathbb{R}\backslash\{0\}$  and $\mathbb{N}_{0}=
\mathbb{N}\backslash\{0\}$.
\
\\
\textbullet \;  $\mathcal{B}_{0} = (e_{1},\dots,e_{n+1})$ the
canonical basis of $\mathbb{C}^{n+1}$  and   $I_{n+1}$  the
identity matrix of $GL(n+1,\mathbb{C})$.
\
\\
For each $m=1,2,\dots, n+1$, denote by:
\\
\textbullet \; $\mathbb{T}_{m}(\mathbb{C})$ the set of matrices
over $\mathbb{C}$ of the form \begin{align}\begin{bmatrix}
  \mu & \ &  \ & 0 \\
  a_{2,1} & \mu &  \ &  \\\
  \vdots &  \ddots & \ddots & \ \\
  a_{m,1} & \dots & a_{m,m-1} & \mu
\end{bmatrix}& \  \label{eq1}
\end{align}
\\
\textbullet \; $\mathbb{T}_{m}^{\ast}(\mathbb{C})$  the group of
matrices of the form (~\ref{eq1}) with $\mu\neq 0$.
\
\\
Let  $r\in \mathbb{N}$ and $\eta
=(n_{1},\dots,n_{r})\in\mathbb{N}^{r}_{0}$ such that $n_{1}+\dots
+ n_{r}=n+1. $ In particular, $r\leq n+1$. Write
\
\\
\textbullet \; \ $\mathcal{K}_{\eta,r}(\mathbb{C}): =
\mathbb{T}_{n_{1}}(\mathbb{C})\oplus\dots \oplus
\mathbb{T}_{n_{r}}(\mathbb{C}).$ In particular if  $r=1$, then
$\mathcal{K}_{\eta,1}(\mathbb{C}) = \mathbb{T}_{n+1}(\mathbb{C})$
and  $\eta=(n+1)$.
\
\\
\textbullet \; $\mathcal{K}^{*}_{\eta,r}(\mathbb{C}): =
\mathcal{K}_{\eta,r}(\mathbb{C})\cap \textrm{GL}(n+1, \
\mathbb{C})$.\
\
\\
\textbullet \; $u_{0} = (e_{1,1},\dots,e_{r,1})\in
\mathbb{C}^{n+1}$ where
 $e_{k,1} = (1,0,\dots,
0)\in \mathbb{C}^{n_{k}}$, for $k=1,\dots, r$. So
$u_{0}\in\{1\}\times\mathbb{C}^{n}$.
\\
\textbullet \; $p_{2}:\mathbb{C}\times
\mathbb{C}^{n}\longrightarrow\mathbb{C}^{n}$ the second projection
defined by $p_{2}(x_{1},\dots,x_{n+1})=(x_{2},\dots,x_{n+1})$.\
\\
 \textbullet \; $e^{(k)} = (e^{(k)}_{1},\dots,  e^{(k)}_{r})\in \mathbb{C}^{n+1}$ where $$e^{(k)}_{j} = \left\{\begin{array}{c}
  0\in \mathbb{C}^{n_{j}}\ \ \mathrm{if}\ \ j\neq k \\
  e_{k,1}\ \ \ \ \ \ \ \ \mathrm{if}\ \ j = k  \\
\end{array}\right. \ \ \ \ \ \  for \ \ every\ \ 1\leq j,\ k\leq r.$$
\\
\textbullet \; $\textrm{exp} :\ \mathbb{M}_{n+1}(\mathbb{C})
\longrightarrow\textrm{GL}(n+1, \mathbb{C})$  is the matrix
exponential map; set $\textrm{exp}(M) = e^{M}$, $M\in
M_{n+1}(\mathbb{C})$.\
\\
\textbullet \; Define the map \  \  \ $\Phi\ :\ GA(n,\
\mathbb{C})\ \longrightarrow\
 GL(n+1,\ \mathbb{C})$ \\
  $$f =(A,a) \ \longmapsto\ \begin{bmatrix}
                     1  & 0 \\
                     a & A
                 \end{bmatrix}$$
We have the following composition formula
 $$\begin{bmatrix}
                     1  & 0 \\
                     a & A
                 \end{bmatrix}\begin{bmatrix}
                     1  & 0 \\
                     b & B
                 \end{bmatrix} = \begin{bmatrix}
                     1  & 0 \\
                     Ab+a & AB
                 \end{bmatrix}.$$
 Then  $\Phi$  is an injective homomorphism of groups.\ Write
  \
\\
\textbullet \; $G=\Phi(\mathcal{G})$, it is an abelian
sub-semigroup of $GL(n+1, \mathbb{C})$.
\\
\textbullet \;  Define the map \ \ \ $\Psi\ :\ MA(n,\ \mathbb{C})\
\longrightarrow\
 M_{n+1}(\mathbb{C})$
$$f =(A,a) \ \longmapsto\ \begin{bmatrix}
                     0  & 0 \\
                     a & A
                 \end{bmatrix}$$
We can see that   $\Psi$  is injective and linear. Hence
$\Psi(MA(n, \mathbb{C}))$ is a vector subspace of
$M_{n+1}(\mathbb{C})$. We prove (see Lemma~\ref{L:10000001}) that
$\Phi$ and $\Psi$ are related by the following property
$$\textrm{exp}(\Psi(MA(n, \mathbb{C})))=\Phi(GA(n, \mathbb{C})).$$

Let consider the normal form of $\mathcal{G}$: By Proposition
~\ref{p:2}, there exists a $P\in \Phi(\textrm{GA}(n, \mathbb{C}))$
and a partition $\eta$ of $(n+1)$ such that $G'=P^{-1}GP\subset
\mathcal{K}^{*}_{\eta,r}(\mathbb{C})\cap\Phi(MA(n,\mathbb{C}))$.
For such a choice of matrix $P$,  we let
\
\\
\textbullet \; $v_{0} = Pu_{0}$. So $v_{0}\in
\{1\}\times\mathbb{C}^{n}$, since
  $P\in\Phi(GA(n, \mathbb{C}))$.
 \\
 \textbullet \; $w_{0} = p_{2}(v_{0})\in\mathbb{C}^{n}$. We have $v_{0}=(1, w_{0})$.\
 \\
\textbullet \; $\varphi=\Phi^{-1}(P)\in MA(n,\mathbb{C})$.
 \\
\textbullet\  $\mathrm{g} = \textrm{exp}^{-1}(G)\cap \left(
P(\mathcal{K}_{\eta,r}(\mathbb{C}))P^{-1}\right)$. If $G\subset
\mathcal{K}^{*}_{\eta,r}(\mathbb{C})$, we have $P=I_{n+1}$ and
$\mathrm{g} = \textrm{exp}^{-1}(G)\cap
\mathcal{K}_{\eta,r}(\mathbb{C})$.
\
\\
\textbullet\  $\mathrm{g}^{1} = \mathrm{g}\cap
\Psi(MA(n,\mathbb{C}))$. It is an additive sub-semigroup of
$M_{n+1}(\mathbb{C})$ (because by  Lemma ~\ref{LL:002},
$\mathrm{g}$ is an additive sub-semigroup of
$M_{n+1}(\mathbb{C})$).
\
\\
\textbullet\  $\mathrm{g}^{1}_{u} = \{Bu: \ B\in
\mathrm{g}^{1}\}\subset \mathbb{C}^{n+1}, \ \
u\in\mathbb{C}^{n+1}.$
\
\\
\textbullet\   $\mathfrak{q} = \Psi^{-1}(\mathrm{g}^{1})\subset
MA(n, \mathbb{C})$.
 Then $\mathfrak{q}$ is an additive sub-semigroup of $MA(n,\mathbb{C})$ and we have $\Psi(\mathfrak{q})=\mathrm{g}^{1}$.
  By Corollary~\ref{r:1}, we have \ $exp(\Psi(\mathfrak{q}))=\Phi(\mathcal{G}).$
\\
\textbullet\  $\mathfrak{q}_{v} =\{f(v),\ \ f\in
\mathfrak{q}\}\subset \mathbb{C}^{n}, \ \ v\in\mathbb{C}^{n}.$
\
\\

For groups of affine maps on  $\mathbb{K}^{n}$
($\mathbb{K}=\mathbb{R}$ or $\mathbb{C}$), their dynamics were
recently initiated for some classes in different point of view,
(see for instance, \cite{Ja}, \cite{Ku}, \cite{be}, \cite{AMY}).
The purpose here is to give analogous results as for linear
abelian sub-semigroup of $GL(n, \mathbb{C})$
 (\cite{aAh-M05}, Theorem 1.1).\
 \\
 \\
  Our main results are the following:

\begin{thm}\label{T:1} Let  $\mathcal{G}$ be an abelian sub-semigroup of $MA(n,\mathbb{C})$. Then  the  following are equivalent:
\
\begin{itemize}
  \item [(i)] $\mathcal{G}$ is hypercyclic.
  \item [(ii)] the orbit $\mathcal{G}(w_{0})$ is dense in $\mathbb{C}^{n}$.
  \item [(iii)] $\mathfrak{q}_{w_{0}}$ is an additive sub-semigroup dense in
$\mathbb{C}^{n}$.
\end{itemize}
\end{thm}
\

For a \textit{finitely generated} abelian sub-semigroup
$\mathcal{G}\subset \textrm{MA}(n, \mathbb{R})$, let introduce the
following property: Consider the following rank condition on a
collection of affine maps $f_{1},\dots,f_{p}\in \mathcal{G}$. Let
  $f'_{1},\dots,f'_{p}\in \mathfrak{q}$ be such that
$e^{\Psi(f'_{k})} = \Phi(f_{k})$, $k=1,\dots,p$.  We say that
$f_{1},\dots,f_{p}$ satisfy the \emph{property $\mathcal{D}$} if
for every $(s_{1},\dots,s_{p};\ t_{2},\dots,t_{r})\in
\mathbb{Z}^{p+r-1}\backslash\{0\}:$   $$\mathrm{rank}
\left[\begin{array}{cccccc}
   \mathrm{Re}(f'_{1}(w_{0})) & \dots& \mathrm{Re}(f'_{p}(w_{0}))& 0 &\dots &0\\
      \mathrm{Im}(f'_{1}(w_{0})) & \dots& \mathrm{Im}(f'_{p}(w_{0}))& 2\pi p_{2}(e^{(2)})&\dots &2\pi p_{2}(e^{(r)}) \\
      s_{1} & \dots& s_{p}& t_{2}&\dots& t_{r} \\
 \end{array}\right]=2n+1.$$\
 For $r=1$, this means that  for every $(s_{1},\dots,s_{p})\in \mathbb{Z}^{p}\backslash\{0\}$: \\
$$\mathrm{rank} \left[\begin{array}{ccc}
   \mathrm{Re}(f'_{1}(w_{0})) & \dots& \mathrm{Re}(f'_{p}(w_{0}))\\
      \mathrm{Im}(f'_{1}(w_{0})) & \dots& \mathrm{Im}(f'_{p}(w_{0}))\\
      s_{1} & \dots& s_{p} \\
 \end{array}\right]=2n+1.$$
\
\\
For a vector $v\in\mathbb{C}^{n}$, we write $v = \mathrm{Re}(v)+
i\mathrm{Im}(v)$ where $\mathrm{Re}(v)$ and $\mathrm{Im}(v)\in
\mathbb{R}^{n}$. The next result can be stated as follows:
\

\begin{thm}\label{T:2} Let  $\mathcal{G}$  be an abelian sub-semigroup of  $MA(n, \mathbb{C})$
and let $f_{1}, \dots,f_{p}\in \mathcal{G}$ generating
$\mathcal{G}^{*}$ and let $f'_{1},\dots,f'_{p}\in\mathfrak{q}$ be
such that $e^{\Psi(f'_{1})} = \Phi(f_{1}),\dots, e^{\Psi(f'_{p})}
= \Phi(f_{p})$. Then the following are equivalent:
\begin{itemize}
  \item [(i)] $\mathcal{G}$ is hypercyclic.
  \item [(ii)] the maps $\varphi^{-1}\circ f_{1}\circ \varphi,\dots,\varphi^{-1}\circ f_{p}\circ \varphi$ in $MA(n, \mathbb{C})$ satisfy the property $\mathcal{D}$.
  \item [(iii)] $\mathfrak{q}_{w_{0}}= \left\{\begin{array}{c}
                                \underset{k=1}{\overset{p}{\sum}}\mathbb{N}f'_{k}(w_{0})+2i\pi\underset{k=2}{\overset{r}{\sum}}\mathbb{Z}(p_{2}(Pe^{(k)})),\ \ \  if\ r\geq 2 \\
\underset{k=1}{\overset{p}{\sum}}\mathbb{N}f'_{k}(w_{0}),\ \   if
\ r=1\ \ \  \ \ \ \ \ \ \ \ \ \ \ \ \ \ \ \ \ \ \ \ \ \ \ \ \ \ \
\end{array}\right.$

is an additive sub-semigroup dense in $\mathbb{C}^{n}$.
\end{itemize}
\end{thm}
\medskip

\begin{cor}\label{C:9} Let  $\mathcal{G}$ be an abelian sub-semigroup of $MA(n,
\mathbb{C})$ and $G=\Phi(\mathcal{G})$. Let $P\in \Phi(GA(n,
\mathbb{C}))$ such that $P^{-1}GP\subset
\mathcal{K}_{\eta,r}(\mathbb{C})$ where $1\leq r \leq n+1$ and
$\eta=(n_{1},\dots,n_{r})\in\mathbb{N}_{0}^{r}$. If \
$\mathcal{G}$ is generated by $2n - r+1$ commuting invertible
affine maps, then it has no dense orbit.
\end{cor}
\medskip

\begin{cor}\label{C:10}  Let  $\mathcal{G}$ be an abelian sub-semigroup of $MA(n,
\mathbb{C})$. If $\mathcal{G}$ is generated by $n$ commuting
invertible affine maps, then it has no dense orbit.
\end{cor}
\medskip

\section{Normal form of abelian affine groups}
\medskip

 \begin{prop}\label{p:2}$($\cite{AAFF}, Proposition 2.1$)$ Let  $\mathcal{G}$ be an abelian subgroup of
$GA(n,\mathbb{C})$ and  $G=\Phi(\mathcal{G})$. Then there exists
$P\in \Phi(GA(n,\mathbb{C}))$ such that $P^{-1}GP$ is a subgroup
of $\mathcal{K}^{*}_{\eta,r}(\mathbb{C})\cap\Phi(GA(n,
\mathbb{C}))$,
 for some $r\leq n+1$ and $\eta=(n_{1},\dots,n_{r})\in\mathbb{N}_{0}^{r}$.
\end{prop}
\medskip

This proposition can be generalized for any abelian affine
semigroup as follow:
\begin{prop}\label{p:0011}   Let  $\mathcal{G}$ be an abelian sub-semigroup of
$MA(n,\mathbb{C})$ and  $G=\Phi(\mathcal{G})$. Then there exists
$P\in \Phi(GA(n,\mathbb{C}))$ such that $P^{-1}GP$ is a subgroup
of $\mathcal{K}^{*}_{\eta,r}(\mathbb{C})\cap\Phi(GA(n,
\mathbb{C}))$,
 for some $r\leq n+1$ and $\eta=(n_{1},\dots,n_{r})\in\mathbb{N}_{0}^{r}$.
\end{prop}

The same proof of Proposition~\ref{p:2} remained valid for
Proposition~\ref{p:0011}.

The group $G'=P^{-1}GP$ is called the  \emph{normal form} of $G$.
In particular we have $Pu_{0}=v_{0}\in\{1\}\times \mathbb{C}^{n}$.
To prove Proposition ~\ref{p:2}, we need the following results:

\
\\
 Denote by $\mathcal{L}_{\mathcal{G}}$  the set of the linear parts of all elements of  $\mathcal{G}$.
  Then $\mathcal{L}_{\mathcal{G}}$ is an abelian sub-semigroup of $M_{n}(\mathbb{C})$. A subset $F\subset \mathbb{C}^{n}$  is called \emph{$
G$-invariant} (resp. \emph{$\mathcal{L}_{\mathcal{G}}$-invariant})
if  $A(F)\subset F$  for any  $A\in G$ (resp. $A\in
\mathcal{L}_{\mathcal{G}}$).
\medskip

\begin{prop}\label{p:1}$($~\cite{aA-Hm12}$)$ Let  $G'$ be an abelian sub-semigroup of \ $M_{m}(\mathbb{C})$, $m\geq 1$. Then
there exists $P\in GL(m,\mathbb{C})$ such that\ $P^{-1}G'P$ is a
sub-semigroup of  $\mathcal{K}_{\eta',r'}(\mathbb{C})$, for some
$r'\leq m$ and
$\eta'=(n'_{1},\dots,n'_{r'})\in\mathbb{N}_{0}^{r'}$.
\end{prop}

\medskip

\begin{lem}\label{L:4-}$($\cite{aAh-M05}, Proposition 3.2$)$ $exp(\mathcal{K}_{\eta,r}(\mathbb{C}))=\mathcal{K}^{*}_{\eta,r}(\mathbb{C})$.
\end{lem}
\medskip

\begin{lem}\label{L:10000001}$($~\cite{AAFF}, Lemma 2.8$)$ $exp(\Psi(MA(n, \mathbb{C}))=GA(n, \mathbb{C})$.
\end{lem}

\medskip

\begin{lem}\label{L:4}$($~\cite{AAFF}, Lemma 2.9$)$  If \  $N\in P\mathcal{K}_{\eta,r}(\mathbb{C})P^{-1}$ such that $e^{N}\in\Phi(GA(n, \mathbb{C}))$,
then there exists $k\in\mathbb{Z}$ such that $N-2ik\pi
I_{n+1}\in\Psi(MA(n,\mathbb{C}))$.
\end{lem}
\medskip

Denote by $G^{*}=G\cap GL(n+1, \mathbb{C})$.

\begin{lem}\label{L:21013}$($\cite{aAh-M05}, Lemma 4.2$)$ One has
$exp(\mathrm{g})=G^{*}$.
\end{lem}
\medskip

\begin{cor}\label{C:21012}$($~\cite{AAFF}, Corollary 2.11$)$ Let $G=\Phi(\mathcal{G})$. We have $\mathrm{g}= \mathrm{g}^{1}+2i\pi \mathbb{Z}I_{n+1}$.\
\end{cor}

Denote by $\mathcal{G}^{*}=\mathcal{G}\cap GA(n, \mathbb{C})$.

\begin{cor}\label{r:1} We have $exp(\Psi(\mathfrak{q}))=\Phi(\mathcal{G}^{*})$.
\end{cor}
\medskip

\begin{proof}
By Lemmas ~\ref{L:21013} and ~\ref{C:21012}, We have
$G=exp(\mathrm{g})=exp(\mathrm{g}^{1}+2i\pi
\mathbb{Z}I_{n+1})=exp(\mathrm{g}^{1})$. Since
$\mathrm{g}^{1}=\Psi(\mathfrak{q})$, we get
$exp(\Psi(\mathfrak{q}))=\Phi(\mathcal{G})$.
\end{proof}
\medskip

\section{ Proof of Theorem ~\ref{T:1}}
\medskip

Let $\widetilde{G}$ be the semigroup generated by $G$ and
$\mathbb{C}I_{n+1}=\{\lambda I_{n+1}:\  \  \ \lambda\in \mathbb{C}
\}$. Then $\widetilde{G}$ is an abelian sub-semigroup of
$GL(n+1,\mathbb{C})$. By Proposition~\ref{p:2}, there exists
$P\in\Phi(GA(n, \mathbb{C}))$ such that $P^{-1}GP$ is a
sub-semigroup of $\mathcal{K}^{*}_{\eta,r}(\mathbb{C})$ for some
$r\leq n+1$ and $\eta=(n_{1},\dots,n_{r})\in\mathbb{N}_{0}^{r}$
and this also implies that $P^{-1}\widetilde{G}P$ is a
sub-semigroup of $\mathcal{K}^{*}_{\eta,r}(\mathbb{C})$. Set
$\widetilde{\mathrm{g}}=exp^{-1}(\widetilde{G})\cap
(P\mathcal{K}_{\eta,r}(\mathbb{C})P^{-1})$ \ and \
 $\widetilde{\mathrm{g}}_{v_{0}}=\{Bv_{0}\ : \ B\in
\widetilde{\mathrm{g}}\}$.\ Then we have the following theorem,
applied to $\widetilde{G}$:
\medskip

\begin{thm}\label{T:5}$($~\cite{aA-Hm12},\ Theorem 1.1$)$  Under the notations above, the following properties are equivalent:
\begin{itemize}
  \item [(i)] $\widetilde{G}$ has a dense orbit in
  $\mathbb{C}^{n+1}$.
  \item [(ii)] the orbit $\widetilde{G}(v_{0})$ is dense in
$\mathbb{C}^{n+1}$.
  \item [(iii)] $\widetilde{\mathrm{g}}_{v_{0}}$ is an additive sub-semigroup dense in
$\mathbb{C}^{n+1}$.
\end{itemize}
\end{thm}
\medskip

\begin{lem}\label{LL:002}$($\cite{aAh-M05},\ Lemma 4.1$)$ The sets $\mathrm{g}$ and \ $\widetilde{\mathrm{g}}$ are additive subgroups of
$M_{n+1}(\mathbb{C})$. In particular, $\mathrm{g}_{v_{0}}$ and
$\widetilde{\mathrm{g}}_{v_{0}}$ are additive subgroups of
$\mathbb{C}^{n+1}$.
\end{lem}
\medskip

Recall that $\mathrm{g}^{1}=\mathrm{g}\cap \Psi(MA(n,
\mathbb{C}))$ and $\mathfrak{q}=\Psi^{-1}(\mathrm{g}^{1})\subset
MA(n, \mathbb{C})$.
\medskip

\begin{lem}\label{L:01234} Under the notations above, one has:\
\begin{itemize}
  \item [(i)] $\widetilde{\mathrm{g}}=\mathrm{g}^{1}+\mathbb{C}I_{n+1}$.\
  \item [(ii)] $\{0\}\times\mathfrak{q}_{w_{0}}=\mathrm{g}^{1}_{v_{0}}$.
 \end{itemize}
\end{lem}
\medskip

\begin{proof} $(i)$ Let $B\in \widetilde{\mathrm{g}}$, then $e^{B}\in \widetilde{G}$. One can write $e^{B}=\lambda A$ for some
 $\lambda\in \mathbb{C}^{*}$ and $A\in G$. Let $\mu\in \mathbb{C}$ such that $e^{\mu}=\lambda$, then $e^{B-\mu I_{n+1}}=A$.
 Since $B-\mu I_{n+1}\in
  P\mathcal{K}_{\eta,r}(\mathbb{C})P^{-1}$, so $B-\mu I_{n+1}\in exp^{-1}(G)\cap P\mathcal{K}_{\eta,r}(\mathbb{C})P^{-1}=\mathrm{g}$.
   By Corollary~\ref{C:21012}, there exists $k\in \mathbb{Z}$ such that
   $B':=B-\mu I_{n+1}+2ik\pi I_{n+1}\in \mathrm{g}^{1}$. Then
  $B\in \mathrm{g}^{1}+\mathbb{C}I_{n+1}$ and hence
   $\widetilde{\mathrm{g}}\subset \mathrm{g}^{1}+\mathbb{C}I_{n+1}$. Since $\mathrm{g}^{1}\subset \widetilde{\mathrm{g}}$
    and   $\mathbb{C}I_{n+1}\subset \widetilde{\mathrm{g}}$, it follows that
    $\mathrm{g}^{1}+\mathbb{C}I_{n+1} \subset \widetilde{\mathrm{g}}$ (since $\widetilde{\mathrm{g}}$
    is an additive group, by Lemma ~\ref{LL:002}). This proves (i).
\
\\
$(ii)$ Since $\Psi(\mathfrak{q})=\mathrm{g}^{1}$ and $v_{0}=(1,
w_{0})$, we obtain for every
 $f=(B,b)\in \mathfrak{q}$, \begin{align*}\Psi(f)v_{0}& =\left[\begin{array}{cc}
                                                                          0 & 0 \\
                                                                          b & B
                                                                        \end{array}
\right]\left[\begin{array}{c}
               1 \\
               w_{0}
             \end{array}
\right]\\ \ & =\left[\begin{array}{c}
               0\\
               b+Bw_{0}
             \end{array}
\right]\\ \ & =\left[\begin{array}{c}
               0\\
               f(w_{0})
             \end{array}
\right].
\end{align*}
Hence $\mathrm{g}^{1}_{v_{0}}=\{0\}\times\mathfrak{q}_{w_{0}}$.

  \end{proof}

\begin{lem}\label{LL0L:1} The following assertions are equivalent:\
\begin{itemize}
  \item [(i)] $\overline{\mathfrak{q}_{w_{0}}}=\mathbb{C}^{n}$.\
  \item [(ii)] $\overline{\mathrm{g}^{1}_{v_{0}}}=\{0\}\times\mathbb{C}^{n}$.\
  \item [(iii)] $\overline{\widetilde{\mathrm{g}}_{v_{0}}}=\mathbb{C}^{n+1}$.\
\end{itemize}
\end{lem}
\medskip

\begin{proof} $(i)\Longleftrightarrow(ii)$  follows from the fact that $\{0\}\times\mathfrak{q}_{w_{0}}=\mathrm{g}^{1}_{v_{0}}$
 (Lemma ~\ref{L:01234},(ii)).
\
\\
$(ii)\Longrightarrow(iii):$ By Lemma ~\ref{L:01234},(ii),
$\widetilde{\mathrm{g}}_{v_{0}}=\mathrm{g}^{1}_{v_{0}}+\mathbb{C}v_{0}$.
    Since $v_{0}=(1,w_{0})\notin \{0\}\times\mathbb{C}^{n}$ and
 $\mathbb{C}I_{n+1}\subset \widetilde{\mathrm{g}}$,
 we obtain $\mathbb{C}v_{0}\subset\widetilde{\mathrm{g}}_{v_{0}}$ and so
 $\mathbb{C}v_{0}\subset\overline{\widetilde{\mathrm{g}}_{v_{0}}}$. Therefore
 $\mathbb{C}^{n+1}=\{0\}\times\mathbb{C}^{n}\oplus \mathbb{C}v_{0}=
 \overline{\mathrm{g}^{1}_{v_{0}}}\oplus \mathbb{C}v_{0}\subset \overline{\widetilde{\mathrm{g}}_{v_{0}}}$
 (since, by Lemma~\ref{LL:002}, $\widetilde{\mathrm{g}}_{v_{0}}$ is an additive sub-semigroup of $\mathbb{C}^{n+1}$).
 Thus $\overline{\widetilde{\mathrm{g}}_{v_{0}}}= \mathbb{C}^{n+1}$.\
 \\
$(iii)\Longrightarrow(ii):$ Let $x\in \mathbb{C}^{n}$, then
$(0,x)\in \overline{\widetilde{\mathrm{g}}_{v_{0}}}$
 and there exists a sequence $(A_{m})_{m\in\mathbb{N}}\subset\widetilde{\mathrm{g}}$ such that $\underset{m\to+\infty}{lim}A_{m}v_{0}=(0,x)$.
 By Lemma~\ref{L:01234}, we can write $A_{m}v_{0}=\lambda_{m}v_{0}+B_{m}v_{0}$ with $\lambda_{m}\in \mathbb{C}$ and
 $B_{m}=\left[\begin{array}{cc}
                0 & 0 \\
                b_{m} & B^{1}_{m}
              \end{array}
 \right]\in\mathrm{g}^{1}$ for every $m\in\mathbb{N}$. Since $B_{m}v_{0}\in\{0\}\times \mathbb{C}^{n}$ for every
 $m\in\mathbb{N}$ then $A_{m}v_{0}=(\lambda_{m},\ b_{m}+B^{1}_{m}w_{0}+\lambda_{m}w_{0})$.
It follows that $\underset{m\to+\infty}{lim}\lambda_{m}=0$ and
$\underset{m\to+\infty}{lim}A_{m}v_{0}=\underset{m\to+\infty}{lim}B_{m}v_{0}=(0,
x)$, thus $(0,x)\in\overline{\mathrm{g}^{1}_{v_{0}}}$. Hence
$\{0\}\times\mathbb{C}^{n}\subset
\overline{\mathrm{g}^{1}_{v_{0}}}$. Since
$\mathrm{g}^{1}\subset\Psi(MA(n, \mathbb{C}))$,
$\mathrm{g}^{1}_{v_{0}}\subset \{0\}\times\mathbb{C}^{n}$ then we
conclude that
$\overline{\mathrm{g}^{1}_{v_{0}}}=\{0\}\times\mathbb{C}^{n}$.
\end{proof}
\medskip

\begin{lem}\label{LL1L:9} Let $x\in\mathbb{C}^{n}$ and $G=\Phi(\mathcal{G})$. The following are equivalent:\
\begin{itemize}
  \item [(i)] $\overline{\mathcal{G}(x)}=\mathbb{C}^{n}$.
  \item [(ii)] $\overline{G(1,x)}=\{1\}\times\mathbb{C}^{n}$.
  \item [(iii)] $\overline{\widetilde{G}(1,x)}=\mathbb{C}^{n+1}$.
\end{itemize}
\end{lem}
\medskip

\begin{proof} $(i)\Longleftrightarrow (ii):$  is obvious since $\{1\}\times\mathcal{G}(x)=G(1,x)$ by construction.
\
\\
$(iii)\Longrightarrow (ii):$  Let $y\in \mathbb{C}^{n}$ and
$(B_{m})_{m}$ be a sequence in $\widetilde{G}$ such that
$\underset{m\to +\infty}{lim}B_{m}(1,x)=(1,y)$. One can write
$B_{m}=\lambda_{m}\Phi(f_{m})$, with $f_{m}\in \mathcal{G}$ and
$\lambda_{m}\in\mathbb{C}^{*}$, thus $B_{m}(1,x)=(\lambda_{m},\ \
\lambda_{m}f_{m}(x))$, so $\underset{m\to
+\infty}{lim}\lambda_{m}=1$. Therefore, $\underset{m\to
+\infty}{lim}\Phi(f_{m})(1,x)=\underset{m\to
+\infty}{lim}\frac{1}{\lambda_{m}}B_{m}(1,x)=(1,y)$. Hence,
$(1,y)\in \overline{G(1,x)}$. \
\\
$(ii)\Longrightarrow (iii):$ Since
$\mathbb{C}^{n+1}\backslash(\{0\}\times
\mathbb{C}^{n})=\underset{\lambda\in
\mathbb{C}^{*}}{\bigcup}\lambda
\left(\{1\}\times\mathbb{C}^{n}\right)$
 and for every $\lambda\in\mathbb{C}^{*}$, $\lambda G(1, x)\subset \widetilde{G}(1,x)$, we get \begin{align*}
\mathbb{C}^{n+1} & =
\overline{\mathbb{C}^{n+1}\backslash(\{0\}\times
\mathbb{C}^{n})}\\ \ & =\overline{\underset{\lambda\in
\mathbb{C}^{*}}{\bigcup}\lambda
\left(\{1\}\times\mathbb{C}^{n}\right)}\\ \ & =
\overline{\underset{\lambda\in \mathbb{C}^{*}}{\bigcup}\lambda
\overline{G(1,x)}} \subset \overline{\widetilde{G}(1,x)}
\end{align*}
Hence $\mathbb{C}^{n+1}=\overline{\widetilde{G}(1, x)}$.
\end{proof}
\medskip

\begin{prop}\label{LL:LL11010a} Let $G$ be an abelian subsemigroup of $M_{n}(\mathbb{C})$ and $G^{*}=G\cap GL(n, \mathbb{C})$. Then $G$
is locally hypercyclic (resp. hypercyclic) if and only if so is
$G^{*}$.
\end{prop}
\smallskip

\begin{proof} Suppose that  $\overset{\circ}{\overline{G^{*}(u)}}\neq \emptyset$, for some $u\in
\mathbb{K}^{n}$. Then
$\emptyset\neq\overset{\circ}{\overline{G^{*}(u)}}\subset
\overset{\circ}{\overline{G(u)}}$ \ and so \
$\overset{\circ}{\overline{G(u)}}\neq \emptyset$. Conversely,
suppose that $\overset{\circ}{\overline{G(u)}}\neq \emptyset$, for
some $u\in\mathbb{C}^{n}$. By proposition 2.1, one can suppose
that $G$ is an abelian sub-semigroup of
$\mathcal{K}_{\eta,r}(\mathbb{C})$. Write
$G^{\prime}:=(G\backslash G^{*})\cup \{I_{n}\}$.  then
$G^{\prime}$ is a sub-semigroup of $G$.

- If \ $G^{\prime}=\{I_{n}\}$  then   $G=G^{*}$ and so $G^{*}$ is
locally hypercyclic.

-  If \ $G^{\prime}\neq \{I_{n}\}$ then

$$G(u)\subset\left(\underset{A\in (G^{\prime}\backslash \{I_{n}\})
}{\bigcup}Im(A)\right)\cup G^{*}(u).$$

As every $A\in (G^{\prime}\backslash\{I_{n}\})$,  is non
invertible, then $Im(A)\subset
\underset{k=1}{\overset{r}{\bigcup}}H_{k}$ where
$$H_{k}:=\left\{u=[u_{1},\dots,u_{r}]^{T}\in \mathbb{C}^{n},\
  u_{j}\in\mathbb{C}^{n_{j}},\  u_{k}\in\{0\}\times\mathbb{C}^{n_{k}-1}
 \begin{array}{c}
  1\leq j \leq r, \\
 j\neq k\
\end{array}\right\}.$$
  It follows that
  $$G(u)\subset\left(\underset{k=1}{\overset{r}{\bigcup}}H_{k}\right)\cup G^{*}(u), $$
and so  $$\overline{G(u)}\subset\left(
\underset{k=1}{\overset{r}{\bigcup}}H_{k}\right)\cup
\overline{G^{*}(u)}.$$

 Since  dim$H_{k}=n-1$, \ $\overset{\circ}{H_{k}}=\emptyset$, for
every \ $1\leq k \leq r$ and therefore \
$\overset{\circ}{\overline{G^{*}(u)}}\neq\emptyset$.
\end{proof}
\medskip

\begin{lem} Let $G$ be an abelian subsemigroup of $\mathcal{K}_{\eta,r}(\mathbb{C})$, $G^{*}=G\cap GL(n, \mathbb{C})$
and $\mathrm{g}^{*} =exp^{-1}(G^{*})\cap
\mathcal{K}_{\eta,r}(\mathbb{C})$. Then
$\mathrm{g}=\mathrm{g}^{*}$.
\end{lem}
\medskip

\begin{proof} Let
$G'=G\backslash G^{*}$. Since $e^{A}\in GL(n, \mathbb{C})$ for
every $A\in M_{n}(\mathbb{C})$ and $G'\subset
M_{n}(\mathbb{C})\backslash GL(n, \mathbb{C})$  then \
$exp^{-1}(G^{*})=\emptyset$. As \
$\mathrm{g}=\left(exp^{-1}(G')\cap
\mathcal{K}_{\eta,r}(\mathbb{C})\right)\cup \mathrm{g}^{*}$ \ then
\ \ $\mathrm{g}=\mathrm{g}^{*}$.
\end{proof}
\medskip
\
\\
{\it Proof of Theorem ~\ref{T:1}.}

By Proposition 2.1, one can suppose that $G$ is an abelian
subsemigroup of $\mathcal{K}_{\eta,r}(\mathbb{C})$. We let $G^{*}
= G\cap GL(n, \mathbb{C})$, $\mathrm{g}^{*}=exp^{-1}(G^{*})\cap
\mathcal{K}_{\eta,r}(\mathbb{C})$ and
$(\mathrm{g}^{*})_{u_{0}}=\{Bu_{0},\ \ B\in \mathrm{g}^{*}\}$. By
applying Proposition 2.7 on $G^{*}$, the orbit $G^{*}(u_{0})$ is
dense in $\mathbb{C}^{n}$ if and only if
$(\mathrm{g}^{*})_{u_{0}}$ is an additive subsemigroup, locally
dense in $\mathbb{C}^{n}$. By Proposition 4.2, Lemma 4.3 and
Proposition 4.6, Theorem 1.1 is proved.

..................................

$(ii)\Longrightarrow(i):$ is obvious.\
\\
$(i)\Longrightarrow(ii):$ Suppose that $\mathcal{G}$ is
hypercyclic, so $\overline{\mathcal{G}(x)}=\mathbb{C}^{n}$ for
some $x\in\mathbb{C}^{n}$. By Lemma~\ref{LL1L:9},(iii),
$\overline{\widetilde{G}(1,x)}=\mathbb{C}^{n+1}$ and by Theorem
~\ref{T:5}, $\overline{\widetilde{G}(v_{0})}=\mathbb{C}^{n+1}$.
Then by Lemma~\ref{LL1L:9},
$\overline{\mathcal{G}(w_{0})}=\mathbb{C}^{n}$, since
$v_{0}=(1,w_{0})$.
\\
$(ii)\Longrightarrow(iii):$ Suppose that
$\overline{\mathcal{G}(w_{0})}=\mathbb{C}^{n}$. By
Lemma~\ref{LL1L:9},
$\overline{\widetilde{G}(v_{0})}=\mathbb{C}^{n+1}$ and by Theorem
~\ref{T:5},
$\overline{\widetilde{\mathrm{g}}_{v_{0}}}=\mathbb{C}^{n+1}$. Then
by Lemma~\ref{LL0L:1},
$\overline{\mathfrak{q}_{w_{0}}}=\mathbb{C}^{n}$.\
\\
$(iii)\Longrightarrow(ii):$ Suppose that
$\overline{\mathfrak{q}_{w_{0}}}=\mathbb{C}^{n}$.\ By
Lemma~\ref{LL0L:1},
 $\overline{\widetilde{\mathrm{g}}_{v_{0}}}=\mathbb{C}^{n+1}$ and by Theorem ~\ref{T:5}, $\overline{\widetilde{G}(v_{0})}=\mathbb{C}^{n+1}$.
 Then by Lemma~\ref{LL1L:9}, $\overline{\mathcal{G}(w_{0})}=\mathbb{C}^{n}$. $\hfill{\Box}$
\
\\
\\
{\it Proof of Corollary ~\ref{C0C:1}.} Assume that \
$\mathcal{G}\subset GL(n, \mathbb{C})$ then take
 $P=\mathrm{diag}(1, Q)$ and $G=\Phi(\mathcal{G})$, then $P^{-1}GP\subset
\mathcal{K}_{\eta,r'+1}(\mathbb{C})$ where $\eta=(1,n'_{1},\dots,
n'_{r'})$. Hence $u_{0}=(1, u'_{0})$, $v_{0}=Pu_{0}=(1,Qu'_{0})$
and thus $w_{0}=Qu'_{0}=v'_{0}$. Every $f=(A,0)\in \mathcal{G}$ is
simply noted $A$. Then for every $A\in \mathcal{G}$,
$\Phi(A)=\mathrm{diag}(1, A)$. We can verify that
$\mathrm{g}^{1}=\{\mathrm{diag}(0, B):\  B\in \mathrm{g}'\}$ where
$\mathrm{g}'=exp^{-1}(\mathcal{G})\cap
Q(\mathcal{K}_{\eta',r'}(\mathbb{C}))Q^{-1}$ and so
$\mathfrak{q}=\Psi^{-1}(\mathrm{g}^{1})=\mathrm{g}'$. Hence the
proof of Corollary~\ref{C0C:1} follows directly from Theorem
~\ref{T:1}.$\hfill{\Box}$.
\medskip

\section{ Finitely  generated  subgroups}
\
\\
Recall the following result proved in \cite{aA-Hm12} which applied
to $G$ can be stated as following:
\medskip

\begin{prop}\label{p:10}$($\cite{aA-Hm12}, Proposition 5.1$)$ Let $G$ be an abelian sub-semigroup of $M_{n}(\mathbb{C})$ such that $G^{*}$ is generated by $A_{1},\dots,A_{p}$
and let  $B_{1},\dots,B_{p}\in \mathrm{g}$ such that  $A_{k} =
e^{B_{k}}$, $k = 1,\dots,p$ and $P\in GL(n+1,\mathbb{C})$
satisfying $P^{-1}GP\subset \mathcal{K}_{\eta,r}(\mathbb{C})$.
Then:
 $$\mathrm{g} =
\underset{k=1}{\overset{p}{\sum}}\mathbb{N}B_{k}+2i\pi\underset{k=1}{\overset{r}{\sum}}\mathbb{Z}PJ_{k}P^{-1}
\ \  \mathrm{and} \ \
 \ \mathrm{g}_{v_{0}} = \underset{k=1}{\overset{p}{\sum}}\mathbb{N}B_{k}v_{0} + \underset{k=1}{\overset{r}{\sum}} 2i\pi \mathbb{Z}Pe^{(k)},$$
where $J_{k} =\mathrm{diag}(J_{k,1},\dots, J_{k,r})$ with \
$J_{k,i}=0\in \mathbb{T}_{n_{i}}(\mathbb{C})$ if $i\neq k$ and
$J_{k,k} = I_{n_{k}}$.
 \end{prop}
\medskip

\begin{prop} \label{p:11} Let  $\mathcal{G}$ be an abelian sub-semigroup of  $GA(n, \mathbb{C})$ such that $\mathcal{G}^{*}$ is generated by
 $f_{1},\dots,f_{p}$ and let  $f'_{1},\dots,f'_{p}\in
\mathfrak{q}$ such that $e^{\Psi(f'_{k})}=\Phi(f_{k})$, $k =
1,..,p$.  Let $P$ be as in Proposition~\ref{p:2}. Then:
 $$\mathfrak{q}_{w_{0}}=\left\{\begin{array}{c}
                                     \underset{k=1}{\overset{p}{\sum}}\mathbb{N}f'_{k}(w_{0}) + \underset{k=2}{\overset{r}{\sum}} 2i\pi \mathbb{Z}p_{2}(Pe^{(k)}),\ if\ r\geq2 \\
                                     \underset{k=1}{\overset{p}{\sum}}\mathbb{N}f'_{k}(w_{0}) , \ if\ r=1\ \ \ \ \ \ \ \ \  \ \ \ \ \ \ \ \ \ \  \ \ \ \ \ \ \ \
                                   \end{array}
\right.$$
 \end{prop}

\begin{proof} Let  $G=\Phi(\mathcal{G})$.  Then  $G$ is
generated by  $\Phi(f_{1}),\dots,\Phi(f_{p})$.\ Apply Proposition
~\ref{p:10} to $G$, $A_{k}=\Phi(f_{k})$,
 $B_{k}=\Psi(f'_{k})\in\mathrm{g}^{1}$,  then  we have
$$\mathrm{g} =
\underset{k=1}{\overset{p}{\sum}}\mathbb{Z}\Psi(f'_{k}) +2i\pi
\mathbb{Z} \underset{k=1}{\overset{r}{\sum}} PJ_{k}P^{-1}.$$

We have
$\underset{k=1}{\overset{p}{\sum}}\mathbb{Z}\Psi(f'_{k})\subset\Psi(MA(n,
\mathbb{C}))$. Moreover, for every $k=2,\dots,r$, $J_{k}\in
\Psi(MA(n, \mathbb{C}))$, hence $PJ_{k}P^{-1}\in
\Psi(MA(n,\mathbb{C}))$, since $P\in \Phi(GA(n, \mathbb{C}))$.
 However,  $mPJ_{1}P^{-1}\notin \Psi(MA(n,\mathbb{C}))$ for every $m\in\mathbb{Z}\backslash\{0\}$, since $J_{1}$ has
  the form $J_{1}=\mathrm{diag}(1, J')$ where $J'\in M_{n}(\mathbb{C})$. As  $\mathrm{g}^{1}=\mathrm{g}\cap \Psi(MA(n, \mathbb{C}))$, then $mPJ_{1}P^{-1}\notin \mathrm{g}^{1}$ for every $m\in\mathbb{Z}\backslash\{0\}$.
 Hence we obtain:
$$\mathrm{g}^{1} =\left\{\begin{array}{c}
                       \underset{k=1}{\overset{p}{\sum}}\mathbb{N}\Psi(f'_{k}) +
\underset{k=2}{\overset{r}{\sum}} 2i\pi \mathbb{Z}PJ_{k}P^{-1}, \
if\ r\geq2 \ \ \  \\
                       \underset{k=1}{\overset{p}{\sum}}\mathbb{N}\Psi(f'_{k}), \ if\ r=1 \ \ \ \ \ \ \ \ \ \ \ \ \ \ \ \ \ \ \ \ \ \ \ \ \ \ \ \
                     \end{array}
\right.$$

Since $J_{k}u_{0}=e^{(k)}$, we get $$\mathrm{g}^{1}_{v_{0}}
=\left\{\begin{array}{c}
                       \underset{k=1}{\overset{p}{\sum}}\mathbb{N}\Psi(f'_{k})v_{0} +
\underset{k=2}{\overset{r}{\sum}} 2i\pi \mathbb{Z}Pe^{(k)}, \ if\
r\geq2 \ \ \ \ \\
                       \underset{k=1}{\overset{p}{\sum}}\mathbb{N}\Psi(f'_{k})v_{0}, \ if\ r=1 \ \ \ \ \ \ \ \ \ \ \ \ \ \ \ \ \ \ \ \ \ \ \ \ \ \ \ \
                     \end{array}
\right.$$

By Lemma~\ref{L:01234},(iii), one has
$\{0\}\times\mathfrak{q}_{w_{0}}=\mathrm{g}^{1}_{v_{0}}$ and
$\Psi(f'_{k})v_{0}=(0,\ f'_{k}(w_{0}))$, so
$\mathfrak{q}_{w_{0}}=p_{2}\left(\mathrm{g}^{1}_{v_{0}}\right)$.
It follows that $$\mathfrak{q}_{w_{0}}=\left\{\begin{array}{c}
                                     \underset{k=1}{\overset{p}{\sum}}\mathbb{N}f'_{k}(w_{0}) + \underset{k=2}{\overset{r}{\sum}} 2i\pi \mathbb{Z}p_{2}(Pe^{(k)}),\ if\ r\geq2 \\
                                     \underset{k=1}{\overset{p}{\sum}}\mathbb{N}f'_{k}(w_{0}) , \ if\ r=1\ \ \ \ \ \ \ \ \  \ \ \ \ \ \ \ \ \ \  \ \ \ \ \ \ \ \
                                   \end{array}
\right.$$ The proof is completed.
\end{proof}
\medskip

Recall the following proposition which was proven in \cite{mW}:
\\

\begin{prop}\label{p:12}$(cf.$ \cite{mW}, $page$  $35)$. Let $F = \mathbb{Z}u_{1}+\dots+\mathbb{Z}u_{p}$ with $u_{k} =
Re(u_{k})+i Im(u_{k})$, where $\mathrm{Re}(u_{k})$,
$\mathrm{Im}(u_{k})\in \mathbb{R}^{n}$, $k = 1,\dots, p$. Then $F$
is dense in $\mathbb{C}^{n}$ if and only if for every
$(s_{1},\dots,s_{p})\in
 \mathbb{Z}^{p}\backslash\{0\}$ :
$$\mathrm{rank}\left[\begin{array}{cccc }
 \mathrm{Re}(u_{1 }) &\dots &\dots & \mathrm{Re}(u_{p }) \\
   \mathrm{Im}(u_{1 }) &\dots &\dots & \mathrm{Im}(u_{p}) \\
  s_{1 } &\dots&\dots & s_{p }
 \end{array}\right] =\ 2n+1.$$
\end{prop}
\
\\
\\
\emph{Proof of Theorem ~\ref{T:2}:} This follows directly from
Theorem ~\ref{T:1}, Propositions ~\ref{p:11} and ~\ref{p:12}.
\
\\
\\
\emph{Proof of Corollary ~\ref{C:9}:} First, by Proposition
~\ref{p:12},  if $F = \mathbb{Z}u_{1}+\dots +\mathbb{Z}u_{m}$,
$u_{k}\in \mathbb{C}^{n}$ with $m \leq 2n$, then $F$ cannot be
dense in $\mathbb{C}^{n}$. Now, by the form of
$\mathfrak{q}_{w_{0}}$ in Proposition~\ref{p:11},
$\mathfrak{q}_{w_{0}}$ cannot be dense in $\mathbb{C}^{n}$ and so
 Corollary ~\ref{C:9} follows by Theorem ~\ref{T:2}.  $\hfill{\Box}$
\
\\
\\
\emph{Proof of Corollary ~\ref{C:10}:} Since  $n\leq 2n-r+1$
(because $r\leq n+1$), Corollary ~\ref{C:10}  follows from
Corollary ~\ref{C:9}. $\hfill{\Box}$

\medskip
\section{ Example}
\bigskip

\begin{exe} \label{exe:2} Let  $\mathcal{G}$\ the sub-semigroup of $GA(2, \mathbb{C})$ generated by
$f_{1}=(A_{1},a_{1})$, $f_{2}=(A_{2},a_{2})$,
$f_{3}=(A_{3},a_{3})$ and $f_{4}=(A_{4},a_{4})$ where
$A_{1}=I_{2}, \ \ a_{1}=(1+i, 0)$,\

$A_{2}=\mathrm{diag}(1, e^{-2+i}),
 \  a_{2}=\left(0, \ 0\right),$ $A_{3} =\mathrm{diag}\left(1,\
e^{\frac{-\sqrt{2}}{\pi}+i\left(\frac{\sqrt{2}}{2\pi}-\frac{\sqrt{7}}{2}\right)}\right)$,
\

  $a_{3} =\left(\frac{-\sqrt{3}}{2\pi}+i\left(\frac{\sqrt{5}}{2}-\frac{\sqrt{3}}{2\pi}\right),0\right),$
$A_{4}=I_{2}, \ \ \  a_{4}=(2i\pi,\ 0).$

Then $\mathcal{G}$ is hypercyclic.
\end{exe}

\begin{proof} First one can check that $\mathcal{G}$ is abelian: $f_{i}f_{j}=f_{j}f_{i}$ for every $i,j=1,2,3,4$. Denote by  $G=\Phi(\mathcal{G})$.
 Then  $G$ is generated by $$\Phi(f_{1})=\left[\begin{array}{ccc}
 1 & 0 & 0 \\
 1+i & 1 & 0\\
  0 & 0 & 1
  \end{array}
\right],\  \Phi(f_{2})=\left[\begin{array}{ccc}
 1 & 0 & 0 \\
 0 & 1 & 0\\
 0 & 0 & e^{-2+i}
 \end{array}
\right],$$ \ $$\Phi(f_{3})=\left[\begin{array}{ccc}
                               1 & 0 & 0 \\
                               \frac{-\sqrt{3}}{2\pi}+i\left(\frac{\sqrt{5}}{2}-\frac{\sqrt{3}}{2\pi}\right) & 1 & 0 \\
                               0 & 0 & e^{\frac{-\sqrt{2}}{\pi}+i\left(\frac{\sqrt{2}}{2\pi}-\frac{\sqrt{7}}{2}\right)}
                              \end{array}
\right],\ \Phi(f_{4})=\left[\begin{array}{ccc}
                               1 & 0 & 0 \\
                              2i\pi & 1 & 0 \\
                               0 & 0 & 1
                             \end{array}
\right].$$

Let $f'_{i}=(B_{i},b_{i})$, $i=1,2,3,4$ where
 $$B_{1}=\mathrm{diag}(0,\ 0)=0, \ \  b_{1}=(1+i,\ 0),$$
$$B_{2}=\mathrm{diag}(0, \ -2+i),\ \ \ \ b_{2}=(0,\ 0),$$
$$B_{3}=\mathrm{diag}\left(0,\ \ \
   \frac{-\sqrt{2}}{\pi}+i\left(\frac{\sqrt{2}}{2\pi}-\frac{\sqrt{7}}{2}\right)\right),\ \  \ \ b_{3}=\left(\frac{-\sqrt{3}}{2\pi}+i\left(\frac{\sqrt{5}}{2}-\frac{\sqrt{3}}{2\pi}\right),\ 0\right),$$
    $$B_{4}=\mathrm{diag}(0,0)=0, \ \ b_{4}=(2i\pi, \ 0).$$
    Then we have $e^{\Psi(f'_{i})}=\Phi(f_{i})$, $i=1,2,3,4$.
\
\\
Here $r=2$, $\eta=(2,1)$, $G$ is an abelian sub-semigroup of
$\mathcal{K}_{(2,1),2}^{*}(\mathbb{C})$.\ We have  $P=I_{3}$,
$\varphi=(I_{2},0)$, $u_{0}=v_{0}=(1,0,1)$, $e^{(2)}=(0,0,1)$ and
$w_{0}=(0,1)$.\ By Proposition ~\ref{p:11}, $\mathfrak{q}_{w_{0}}=
\underset{k=1}{\overset{4}{\sum}}\mathbb{Z}f'_{k}(w_{0}) +  2i\pi
\mathbb{Z}p_{2}(e^{(2)})$. On the other hand, for every
$(s_{1},s_{2}, s_{3}, s_{4},
t_{2})\in\mathbb{Z}^{5}\backslash\{0\}$, write $$M_{(s_{1}, s_{2},
s_{3}, s_{4}, t_{2})}=\left[\begin{array}{ccccc }
 \mathrm{Re}(B_{1}w_{0}+b_{1}) & \mathrm{Re}(B_{2}w_{0}+b_{2}) & \mathrm{Re}(B_{3}w_{0}+b_{3}) & \mathrm{Re}(B_{4}w_{0}+b_{4}) & 0 \\
 \\
 \mathrm{Im}(B_{1}w_{0}+b_{1}) & \mathrm{Im}(B_{2}w_{0}+b_{2}) & \mathrm{Im}(B_{3}w_{0}+b_{3}) & \mathrm{Im}(B_{4}w_{0}+b_{4}) & 2\pi e^{(2)} \\
 \\
  s_{1 } & s_{2 } & s_{3 } & s_{4 } & t_{2 }
  \end{array}\right].$$ Then the
determinant:

\begin{align*}
\Delta & = \mathrm{det}\left(M_{(s_{1}, s_{2}, s_{3}, s_{4},
t_{2})}\right)
\\ \ & =\left|\begin{array}{ccccc }
  1 & 0 & -\frac{\sqrt{3}}{2\pi } & 0 & 0 \\
  0 & -2 &  -\frac{\sqrt{2}}{\pi } & 0 & 0 \\
   1 & 0 & \frac{\sqrt{5}}{2}-\frac{\sqrt{3}}{2\pi } & 2\pi & 0 \\
   0 & 1 & \frac{\sqrt{2}}{2\pi}-\frac{\sqrt{7}}{2 } & 0 & 2\pi \\
  s_{1 } & s_{2 } & s_{3 } & s_{4 } & t_{2 }
  \end{array}\right|\\
  \ & =2\pi\left(-s_{1}\sqrt{3}+2s_{2}\sqrt{2}-4s_{3}\pi+s_{4}\sqrt{5}-t_{2}\sqrt{7}\right).
  \end{align*}

Since $\pi$, $\sqrt{2}$, $\sqrt{3}$, $\sqrt{5}$ and $\sqrt{7}$ \ \
are rationally independent,  $\Delta\neq 0$ for every
 $(s_{1}, s_{2}, s_{3}, s_{4}, t_{2})\in\mathbb{Z}^{5}\backslash\{0\}$. It follows that $\mathrm{rank}\left(M_{(s_{1}, s_{2}, s_{3}, s_{4}, t_{2})}\right)=\ 5.$
  Hence $f_{1},\dots,f_{4}$ satisfy the  property $\mathcal{D}$. By Theorem ~\ref{T:2}, $\mathcal{G}$ is hypercyclic.
\end{proof}

\bibliographystyle{amsplain}
\vskip 0,4 cm

\end{document}